\documentclass{amsart}
\usepackage[utf8]{inputenc}
\usepackage[margin=1in]{geometry}
\usepackage{amsmath, amssymb, amsthm, verbatim, hyperref, xcolor,multirow,fancyvrb, tikz-cd,ytableau}
\usetikzlibrary{arrows}

\usepackage{thmtools}
\usepackage{thm-restate}

\usepackage{cellspace} %
\declaretheorem[name=Theorem,numberwithin=section]{thm}
\newtheorem{lemma}[thm]{Lemma}
\newtheorem{prop}[thm]{Proposition}
\newtheorem{cor}[thm]{Corollary}
\newtheorem{conj}[thm]{Conjecture}
\newtheorem{question}[thm]{Question}

\theoremstyle{definition}

\newtheorem{construction}[thm]{Construction}
\newtheorem{definition}[thm]{Definition}
\newtheorem{example}[thm]{Example}

\theoremstyle{remark}
\newtheorem{remark}[thm]{Remark}

\newcommand{\kk}{\mathbf{k}}
\newcommand{\Z}{\mathbb{Z}}
\newcommand{\fS}{\mathfrak{S}}
\newcommand{\N}{\mathbb{N}}
\newcommand{\F}{\mathbb{F}}
\newcommand{\qq}{\mathfrak{q}}

\newcommand{\mm}{\mathfrak{m}}

\newcommand{\cG}{\mathcal{G}}
\renewcommand{\aa}{\mathfrak{a}}

\newcommand{\hgt}{\operatorname{ht}}
\newcommand{\Spec}{\operatorname{Spec}}
\newcommand{\id}{\operatorname{id}}

\renewcommand{\Im}{\operatorname{Im}}
\newcommand{\Tr}{\operatorname{Tr}}
\newcommand{\spn}{\operatorname{span}}
\newcommand{\Hom}{\operatorname{Hom}}
\newcommand{\GL}{\operatorname{GL}}
\newcommand{\rad}{\operatorname{rad}}

\newcommand{\leqdom}{\leq_{\text{dom}}}
\newcommand{\geqdom}{\geq_{\text{dom}}}
\newcommand{\lessdom}{<_{\text{dom}}}

\newcommand{\into}{\hookrightarrow}

\newcommand{\transfer}[1]{\operatorname{Tr}^{#1}}
\newcommand{\imtransfer}[1]{I^{#1}}

\title{Symmetric Group Fixed Quotients of Polynomial Rings}
\author{Alexandra Pevzner}
\date{\today}

\begin{document}

\begin{abstract}
    Given a representation of a finite group $G$ over some commutative base ring $\mathbf{k}$, the \textit{cofixed space} is the largest quotient of the representation on which the group acts trivially. If $G$ acts by $\mathbf{k}$-algebra automorphisms, then the cofixed space is a module over the ring of $G$-invariants. When the order of $G$ is not invertible in the base ring, little is known about this module structure. We study the cofixed space in the case that $G$ is the symmetric group on $n$ letters acting on a polynomial ring by permuting its variables. When $\mathbf{k}$ has characteristic 0, the cofixed space is isomorphic to an ideal of the ring of symmetric polynomials in $n$ variables. Localizing $\mathbf{k}$ at a prime integer $p$ while letting $n$ vary reveals striking behavior in these ideals. As $n$ grows, the ideals stay stable in a sense, then jump in complexity each time $n$ reaches a multiple of $p$. 
\end{abstract}

\maketitle


\section{Introduction}

Fix a commutative ring $\kk$ with unit. Given a representation $U$ of a finite group $G$ over $\kk$, there are two natural $\kk G$-modules one can associate to $U$ on which $G$ acts trivially - the fixed space $U^G$ and the cofixed space $U_G$. The fixed space is the largest $\kk$-submodule of $U$ carrying trivial $G$-action, while the cofixed space is the largest $\kk$-module quotient of $U$ carrying trivial $G$-action. As $\kk$-modules, the fixed space and the cofixed space are nearly dual to each other, with $(U_G)^*\cong (U^*)^G$\cite[Lemma 2.21]{AguiarMahajan}. The functors $(-)^G$ and $(-)_G$ on $\kk G$-modules form an adjoint pair.

When $U$ is a $\kk$-algebra $S$ and $G$ acts on $S$ by $\kk$-algebra automorphisms, an asymmetry between $S^G$ and $S_G$ is apparent. The fixed space $S^G$ is itself a ring, called the \textit{ring of invariants}, and its algebraic structure has been a central object of study in commutative algebra and representation theory for many years. The cofixed space, on the other hand, is not a ring, but it is a module over the ring of invariants.

The structure of the cofixed space as a module over $S^G$ depends greatly on whether or not $|G|$ is invertible in $\kk$. In the nonmodular case, i.e. when $|G|$ is a unit in $\kk$, the cofixed space is a free $S^G$-module of rank one. When $|G|$ is not a unit, very little is known about $S_G$ as an $S^G$-module. In \cite{LewisReinerStanton}, Lewis, Reiner, and Stanton prove that $S_G$ still has rank one over $S^G$, and they give conjectures for the Hilbert series of $\F_q[x_1,\ldots,x_n]_G$ when $G$ is $\GL_n(\F_q)$ or one of its parabolic subgroups.

In this paper, we study the $S^G$-module structure of $S_G$ when $S = \kk[x_1,\ldots,x_n]$, $G = \fS_n$, and $\kk = \Z_{(p)}$ or $\kk = \Z/p\Z$. Here, $\Z_{(p)}$ denotes the localization of $\Z$ at the prime ideal $(p)$. The action of $\fS_n$ is modular for these $\kk$ when $n \geq p$. It is well-known that regardless of $\kk$, the ring of $\fS_n$-invariants is a polynomial ring $\kk[e_1,\ldots,e_n]$, where $e_i$ is the degree $i$ elementary symmetric polynomial in $x_1,\ldots,x_n$. Our main theorem explicitly describes the structure of $S_G$ as a module over this polynomial ring when $p\leq n < 2p$. 

\begin{restatable}[Main Theorem]{thm}{mainthm}
\label{thm: main theorem intro}
    Let $p \leq n < 2p$ and let $i=n-p$. Then the cofixed space $\Z_{(p)}[x_1,\ldots,x_n]_{\fS_n}$ is isomorphic to the ideal $I^{\fS_n}$ of $\Z_{(p)}[e_1,\ldots,e_n]$ given by
    $$I^{\fS_n} = \langle p, \,\,e_1e_p - e_{p+1},\,\, e_2e_p-e_{p+2},\,\,\ldots,\,\,e_ie_p-e_{p+i},\,\,e_{i+1},\,\,\ldots,\,\,e_{p-1}\rangle.$$
    For $n$ in this range, the generators of $I^{\fS_n}$ form a regular sequence and have degrees $0,1,\ldots,p-1 \mod p$.
\end{restatable}

The ideal $I^{\fS_n}$ is exactly the image of the transfer map $\transfer{\fS_n}:S \to S^{\fS_n}$, which sends a polynomial $f\in S$ to the sum $\sum_{\sigma\in \fS_n}\sigma(f)$. The image of this map for arbitrary $G$ has been a longstanding object of interest in the study of modular invariant theory; see \cite{NeuselSmithBook} for background, and \cite{CampbellHughesShankWehlau},\cite{Neusel1},\cite{neusel2},\cite{ShankWehlau} for work on the image of the transfer map for various subgroups $G$ of $\GL_n(\kk)$. This paper also serves to give a description of the image of the transfer map when $p\leq n < 2p$, $G=\fS_n$, and $\kk = \Z_{(p)}, \Z/p\Z$.

To illustrate Theorem \ref{thm: main theorem intro}, we write the ideals $I^{\fS_n}$ of $\Z_{(5)}[e_1,\ldots,e_n]$ for $p=5$ and $n\in \{5,6,7,8,9\}$:
\begin{align*}
    I^{\fS_5} &= \langle 5,\, {\color{red}e_1},\,e_2,\,e_3,\,e_4\rangle \\
    I^{\fS_6} &= \langle 5,\, {\color{blue}e_1e_5-e_6},\,{\color{red}e_2},\,e_3,\,e_4\rangle \\
    I^{\fS_7} &= \langle 5,\, {\color{blue}e_1e_5-e_6},\,{\color{blue}e_2e_5-e_7},\,
    {\color{red}e_3},\,e_4\rangle \\
    I^{\fS_8} &= \langle 5,\, {\color{blue}e_1e_5-e_6},\,{\color{blue}e_2e_5-e_7},\,{\color{blue}e_3e_5-e_8},\,{\color{red}e_4}\rangle \\
    I^{\fS_9} &= \langle 5,\, {\color{blue}e_1e_5-e_6},\,{\color{blue}e_2e_5-e_7},\,{\color{blue}e_3e_5-e_8},\,{\color{blue}e_4e_5-e_9}\rangle.
\end{align*}

The ideals $I^{\fS_n}$ follow the pattern that $I^{\fS_{p+j}}$ can be obtained from $I^{\fS_{p+j-1}}$ by replacing the $e_j$ generator with $e_je_p-e_{p+j}$. Moreover, the ideals are always minimally generated by a regular sequence with degrees of generators being the same mod $p$. This means that the minimal free resolutions of $I^{\fS_n}$ over $S^{\fS_n}$ have the same structure, and the graded Betti numbers\footnote{See Appendix \ref{section: appendix} for notes on why graded minimal free resolutions over the polynomial ring $\Z_{(p)}[e_1,\ldots,e_n]$ are unique up to isomorphism. This allows us to use the term ``Betti number" unambiguously.} $\beta^{S^{\fS_n}}_{i,j}(I^{\fS_n})$ are the same mod $p$. This stability in the module structure of $S_G$ is trivially true when $0\leq n < p$: since $S_G$ is free of rank 1 over $S^G$, the cofixed space has one $S^G$-generator in degree 0, and no syzygies. When $n \geq 2p$, the structure of $I^{\fS_n}$ becomes more complicated, and $\Z_{(p)}[e_1,\ldots,e_n]/I^{\fS_n}$ is no longer a complete intersection ring. However, the stability of graded Betti numbers mod $p$ persists. As an example, below we show the Betti tables for $I^{\fS_4}$ and $I^{\fS_5}$ over $\Z_{(2)}[e_1,\ldots,e_4]$ and $\Z_{(2)}[e_1,\ldots,e_5]$, respectively.
\begin{center}
\begin{minipage}[c]{0.3\textwidth}
\begin{Verbatim}[commandchars=\\\{\}]
       0 1 2 3 
total: 5 7 4 1 
    0: 1 1 . . 
    1: 1 1 1 . 
    2: 1 2 1 . 
    3: 1 1 1 1 
    4: . 1 1 . 
    5: . 1 . . 
    6: 1 . . . 
\end{Verbatim}
\end{minipage}
\begin{minipage}[c]{0.3\textwidth}
\begin{Verbatim}[commandchars=\\\{\}]
        0 1 2 3 
total:  5 7 4 1 
    0:  1 . . . 
    1:  . 1 . . 
    2:  1 1 . . 
    3:  1 . 1 . 
    4:  . 2 . . 
    5:  1 . 1 . 
    6:  . 1 1 . 
    7:  . 1 . 1 
    8:  . . 1 . 
    9:  . 1 . . 
    10: 1 . . . 
\end{Verbatim}
\end{minipage}
\end{center}

\bigskip

Just as in the case that $n < 2p$, the minimal generators have the same degrees mod $2$ (in fact, they are the same mod $4$), and the degree shifts appearing in higher homological degree are also the same mod $4$. Along with more data to this effect, these findings suggest the following conjecture.

\begin{conj}\label{conj: main conj intro}
    Let $k, m, n$ be nonnegative integers such that $kp \leq m,n < (k+1)p$. Let $I^{\fS_m}\subset \Z_{(p)}[e_1,\ldots,e_m]$ and $I^{\fS_n}\subset\Z_{(p)}[e_1,\ldots,e_n]$ be the $\fS_m$- and $\fS_n$-cofixed spaces of the polynomial ring in $m,n$ variables, respectively. Then for a fixed $i\geq 0$, the $i^{\text{th}}$ columns of the Betti tables for $I^{\fS_m}$ and $I^{\fS_m}$ have the same degree shifts mod $p$, with multiplicity.
\end{conj}

\subsection{Organization of paper}
In Section \ref{section: background}, we discuss background on cofixed spaces of finite group representations and cite general results on the $S^G$-module structure of the cofixed space. In Section \ref{section: permutation groups and transfer map}, we discuss the modular transfer map and what is known about its image. We then relate the transfer map back to the cofixed space over rings of characteristic 0 and prove a useful base change lemma (Lemma \ref{lem: base change}). Section \ref{section: main theorem proof} goes through the steps necessary to prove Theorem \ref{thm: main theorem intro}. Of particular importance is the relationship between the structure of the Sylow $p$-subgroups of $\fS_n$ and the stability in the ideals $I^{\fS_n}$; this is the topic of subsection \ref{subsection: p-Sylow}. Theorem \ref{thm: main theorem intro} is proven in subsection \ref{subsection: main proof}. We discuss applications of Theorem \ref{thm: main theorem intro} to the cofixed space and the transfer map with $\Z/p\Z$ coefficients in Section \ref{section: applications}. In Section \ref{section: conjecture}, we provide data in support of Conjecture \ref{conj: main conj intro}, which is restated more precisely as Conjecture \ref{conj: betti number}. Appendix \ref{section: appendix} is dedicated to verifying that graded minimal free resolutions over polynomial rings with local coefficient rings are unique up to isomorphism.

\section*{Acknowledgments}
The author is very grateful to Victor Reiner for introducing this project and for continued guidance during all of its stages. The author would also like to thank Lukas Katth\"an for initial work on the project, and Ayah Almousa, Eddy Campbell, Patricia Klein, Monica Lewis, Andrew O'Desky, Michael Perlman, McCleary Philbin, Mahrud Sayrafi, Gregory Smith, David Wehlau, and Jerzy Weyman for helpful references and conversations. All computations related to this paper were done in \texttt{Macaulay2} and relied on the \texttt{InvariantRing}, \texttt{LocalRings}, and \texttt{SymmetricPolynomials} packages. The author was partially supported by NSF grant DMS-2053288.

\section{Background on cofixed spaces}\label{section: background}

\subsection{Cofixed spaces of general representations}

\begin{definition}\label{def: cofixed space}
    Let $U$ be a finitely-generated free $\kk$-module and let $G$ be a finite group acting on $U$ via $\kk$-linear automorphisms of $U$. We define the \textit{fixed space} $U^G$, and the \textit{cofixed space} $U_G$, respectively, to be the $\kk$-modules
    \begin{align*}
        U^G &:= \left\{ u \in U : g(u) = u \text{ for all }g\in G\right\}, \\
        U_G &:= U / \spn_\kk\left\{ u - g(u) : u \in U, g \in G \right\}.
    \end{align*}
\end{definition}

The fixed and cofixed spaces satisfy the properties that $U^G$ is the largest submodule of $U$ on which $G$ acts trivially, and $U_G$ is the largest quotient of $U$ on which $G$ acts trivially. They can also be defined as
\begin{align*}
    U^G &= \Hom_{\kk G}(\kk, U), \quad \text{and}\\
    U_G &= \kk\otimes_{\kk G}U,
\end{align*}
where $\kk$ is the trivial $\kk G$-module. The group homology and group cohomology functors $H_i(G,-)$ and $H^i(G,-)$, respectively, are the left- and right-derived functors of $(-)_G$ and $(-)^G$. The fixed space is well-known and well-studied due to its importance in the invariant theory of finite groups, when $U$ is a ring. The cofixed space appears less often in the literature, but is sometimes used to define other objects. See, e.g. \cite[\S 3.1]{ChurchEllenbergFarb} for its role in defining the stability degree of an FI-module and \cite[Ch.15]{AguiarMahajan} for its role in defining the bosonic Fock functors. We now provide two specific examples of fixed and cofixed spaces so that the reader can gain intuition.

\begin{example}
    When $U = \kk[x_1,\ldots,x_n]$ and $G$ is a subgroup of $\fS_n$ which acts by permuting variables, then $U^G$ and $U_G$ are isomorphic free $\kk$-modules, with $\kk$-bases
    \begin{align*}
        U^G &= \spn_\kk\left\{\sum_{\sigma\in G/G_\lambda}x^\lambda:\lambda\in\Lambda\right\}, \\
        U_G &= \spn_\kk\left\{\overline{x^\lambda}:\lambda\in\Lambda\right\}
    \end{align*}
    where $\Lambda$ is a complete set of $G$-orbit representatives of monomials in $U$ and $G_\lambda$ denotes the stabilizer subgroup of the monomial $x^\lambda$. That is, the fixed space has a $\kk$-basis of orbit \textit{sums} of monomials, while the cofixed space has a $\kk$-basis of orbit \textit{representatives} of monomials. When $G = \fS_n$, the set $\Lambda$ can be taken to be all integer vectors $(\lambda_1,\ldots,\lambda_n)$ with $\lambda_1\geq \cdots \geq \lambda_n\geq 0$.
\end{example}

While it is often true that $U^G$ and $U_G$ are isomorphic as $\kk$-modules, this is not always the case, as we show in the example below.

\begin{example}
    Let $V = \F_3[x,y]$ and let $G=\GL_2(\F_3)$ act by ring automorphisms of $V$ induced from the action of $G$ on the $\F_3$-vector space with basis $x,y$. The action of $G$ preserves the standard grading on $V$. We consider $U = V_4$, the $\F_3$-span of the degree 4 elements in $V$, which is generated as an $\F_3$-vector space by all monomials of degree 4. For convenience, we list a generating set for $G$:
    
    $$A = \begin{bmatrix} -1 & 0 \\ 0 & 1 \end{bmatrix}, \quad B = \begin{bmatrix} 1 & 0 \\ 0 & -1 \end{bmatrix}, \quad C= \begin{bmatrix} 0 & 1 \\ 1 & 0 \end{bmatrix}, \quad D = \begin{bmatrix} 1 & 1 \\ 0 & 1\end{bmatrix}.$$
    
    Any element in $U^G$ must lie in the $\F_3$-span of $\{x^4, x^2y^2,y^4\}$ in order to be invariant under $A$ and $B$, and furthermore must be in the $\F_3$-span of $\{x^4+y^4,x^2y^2\}$ to be invariant under $C$. One can directly show that any $\F_3$-linear combination $a(x^4+y^4)+b(x^2y^2)$ which is invariant under $D$ must satisfy $a=b=0$. Hence, $U^G = 0$.
    
    On the other hand, $U_G$ is a 1-dimensional $\F_3$-vector space. The matrices $A$ and $B$ give the relations $2x^3y=0$ and $2xy^3=0$ in $U_G$. The matrix $C$ gives the relation $x^4 - y^4=0$, and combining with $y^4 = (x+y)^4$ gives $x^4=y^4=0$. However, all generating matrices and their inverses applied to $x^2y^2$ induce the trivial relation $x^2y^2-x^2y^2=0$ in $U_G$ after modding out by the relations involving the other monomials. We will see in Proposition \ref{prop: LRS 5.1} that it is enough to check the image of $x^2y^2$ on the generators of the group and their inverses. Hence $U_G$ is 1-dimensional, spanned over $\F_3$ by the image of $x^2y^2$.
\end{example}


\subsection{The cofixed space as a module over the ring of invariants}

When $U$ is a $\kk$-algebra $S$ and $G$ acts by $\kk$-algebra automorphisms, the invariant space $S^G$ is a subring of $S$, which we call the \textit{ring of invariants} or the \textit{invariant ring}. The multiplication action of $S^G$ on $S$ descends to the quotient $S_G$, since given $r\in S^G$, $f\in S$, and $g\in G$, we have
$$r(f-g(f)) = rf - rg(f) = rf - g(rf).$$
The cofixed space is therefore a module over the ring of invariants. The following proposition, found in \cite[Proposition 5.1]{LewisReinerStanton}, outlines some basic information on generating $S_G$ over $S^G$. From this it follows that $S_G$ is a finitely generated $S^G$-module, since $S$ is finitely generated over $S^G$ \cite[Theorem 1.3.1]{Benson}.

\begin{prop}\label{prop: LRS 5.1}
Let $M$ be a module over a $\kk$-algebra $R$, and let $G$ be a finite group acting on $M$ by $R$-module automorphisms. Write $M_G = M/N$, where $N = \spn_\kk\{m-g(m):g\in G,m\in M\}$. Suppose that $\{m_i:i\in I\}\subseteq M$ generates $M$ as an $R$-module and that $\{g_j:j\in J\}\subseteq G$ generates $G$ as a group. Then,
\begin{enumerate}
    \item[(i)] the images $\{\overline{m}_i:i\in I\}$ generate $M_G$ as an $R$-module, and
    \item[(ii)] the set $\{m_i - g_j(m_i), m_i - g_j^{-1}(m_i):i\in I, j\in J\}$ generates $N$ as an $R$-module.
\end{enumerate}
\end{prop}

Aside from finite generation of $S_G$ over $S^G$, the rank of $S_G$ as an $S^G$-module is also known, due to Lewis, Reiner, and Stanton \cite[Proposition 5.7]{LewisReinerStanton}.

\begin{prop}\label{prop: rank one module}
    Let $S$ be a $\kk$-algebra and an integral domain on which a finite group $G$ acts by $\kk$-algebra automorphisms. Then the cofixed space $S_G$ has rank one as a module over the ring of invariants.
\end{prop}

The structure of $S_G$ over $S^G$ becomes very simple when $|G|$ is a unit in $\kk$. To see this, we use the Reynolds operator $\pi^G:S \to S^G$, defined by
$$\pi^G(f) = \frac{1}{|G|}\sum_{g\in G}g(f).$$
The Reynolds operator is a map of $S^G$-modules, and it is a projection onto the ring of invariants whenever it is defined.

\begin{prop}\label{prop: nonmodular cofixed space}
    Suppose $|G|$ is a unit in $\kk$ and that $G$ acts on a $\kk$-algebra $S$ which is an integral domain as in Proposition \ref{prop: rank one module}. Then the cofixed space $S_G$ is a free $S^G$-module of rank one.
\end{prop}

\begin{proof}
    The map $\pi^G$ factors through $S_G$ since any two elements of $S$ in the same $G$-orbit have the same image under $\pi^G$. The induced map $S_G\to S^G$ remains surjective. The map is also injective, since given $f\in S^G$, any two preimages $h, h'\in S$ of $f$ under $\pi^G$ must be in the same $G$-orbit, hence equal in $S_G$.
\end{proof}

\begin{remark}
    Aguiar and Mahajan note in \cite[Lemma 2.20]{AguiarMahajan} that $S_G$ is also a free rank one $S^G$-module when $S$ is a flat $\kk G$-module; in this case the map $|G|\pi^G$ is an isomorphism.
\end{remark}

When $|G|$ is not invertible in $\kk$ and $S$ is not flat over $\kk G$, there is very little known about the structure of $S_G$ as a module over $S^G$.

To give the reader a sense of how the $S^G$ action can be nontrivial when $|G|$ is not invertible in $\kk$, we work out an example below.

\begin{example}\label{ex: p2n3}
Let $S = \F_2[x_1,x_2,x_3]$ and let $G=\fS_3$ act by variable permutation; here $\F_2$ is the field of 2 elements. The invariant ring is a polynomial ring $\F_2[e_1,e_2,e_3]$, where $e_i$ is the degree $i$ elementary symmetric polynomial. By Proposition \ref{prop: LRS 5.1}, a generating set for $S$ over $S^{\fS_3}$ descends to a generating set for $S_{\fS_3}$ over $S^{\fS_3}$. One well-known basis for $S$ over $S^{\fS_3}$ is the set of ``sub-staircase" monomials, also known as the Artin basis \cite{Artin}; these are the monomials $x_1^{\lambda_1}x_2^{\lambda_2}x_3^{\lambda_3}$ which satisfy $0\leq \lambda_i \leq 3-i$ for all $i$. Because all monomials in the same $\fS_3$-orbit are equal in the cofixed space, it suffices to take such monomials with weakly decreasing exponent vectors. This means that the images of $\{1, x_1,x_1^2,x_1x_2,x_1^2x_2\}$ generate the cofixed space over $S^{\fS_3}$. This is not a minimal generating set; notice that
$$e_1\cdot \overline{1} = \overline{x_1} + \overline{x_2} + \overline{x_3} = 3\,\overline{x_1} = \overline{x_1},$$
so $\overline{x_1}$ is redundant. One can similarly show that $\overline{x_1x_2} = e_2\cdot\overline{1}$ and $\overline{x_1^2} = e_1^2\cdot\overline{1}$. However, $\overline{x_1^2x_2}$ is a minimal generator, as any degree three monomial in $e_1,e_2,e_3$ has an even number of terms in the $\fS_3$-orbit of $x_1^2x_2$. Hence, $\{\overline{1},\overline{x_1^2x_2}\}$ is a minimal generating set for $S_{\fS_3}$ over $S^{\fS_3}$. One can compute that $\operatorname{Ann}_{S^{\fS_3}}(\overline{1}) = \langle e_1e_2 + e_3\rangle$, and that $\overline{x_1^2x_2}$ generates a free $S^{\fS_3}$-submodule of $S_{\fS_3}$. Hence, the cofixed space has a decomposition
$$\F_2[x_1,x_2,x_3]_{\fS_3} \cong \frac{\F_2[e_1,e_2,e_3]}{\langle e_1e_2 + e_3\rangle} \oplus \langle e_1e_2+e_3\rangle$$
as a module over $\F_2[e_1,e_2,e_3]$. This decomposition will also follow from Theorem \ref{thm: F_p structure}.
\end{example}

In the next section, we focus on the case when $G$ acts on a polynomial ring $S$ by permuting variables. In this case, we can use the transfer map, a multiple of the Reynolds operator which is defined for all $\kk$, to study $S_G$ as an $S^G$-module. Significantly more is known about the image of the transfer map when $G$ is a permutation group, and this turns out to be closely related to the study of the cofixed space.

\section{The cofixed space of a permutation representation}\label{section: permutation groups and transfer map}

\subsection{The transfer map}

We specialize to the case when $S = \kk[x_1,\ldots,x_n]$ for the remainder of the paper. We give $S$ the standard grading with $\deg(x_i) = 1$ for all $i$ and consider group actions which preserve the graded structure. In other words, $G$ is a subgroup of $\GL(V)$ acting on a vector space $V\cong \kk^n$ with basis $x_1,\ldots,x_n$, and this action extends to $\operatorname{Sym}(V)\cong S$.

\begin{definition}\label{def: transfer map}
    With $G, S$ as above, define the \textit{transfer map} $\Tr_\kk^G: S \to S^G$ by
        $$\Tr_\kk^G(f) = \sum_{g\in G} g(f).$$
    Because $\Tr_\kk^G$ is a map of $S^G$-modules, its image is an ideal of $S^G$, which we denote by $I_\kk^G$. When $\kk$ is clear from context, we may drop the subscript and denote the transfer map by $\Tr^G$ and its image by $I^G$.
\end{definition}

\begin{remark}
    The transfer map descends to the quotient $S_G$. It also can be defined in greater generality for $G$ any finite group and $U$ a representation of $G$ over $\kk$. This gives a natural map $H_0(G, U) \to H^0(G,U)$, sometimes called the \textit{norm map}, and it is used in the study of Tate cohomology; see \cite[\S6.4]{Brown}.
\end{remark}

Unlike the Reynolds operator, the transfer map is defined for $\kk$ of arbitrary characteristic. When $~\operatorname{char}(\kk)=0$ but $|G|\notin \kk^\times$, it is easy to see that $I^G$ is a proper, nonzero ideal of $S^G$; namely $1\notin I^G$ but $|G|\in I^G$. When $\kk$ is a field of characteristic $p$, it is more subtle to see that $I^G$ is a nonzero, proper ideal of $S^G$; the proof in \cite[Theorem 2.2]{ShankWehlau} requires the assumption that the action of $G$ on $S$ comes from a faithful representation $G\into \GL(V)$.

The image of a single element $f$ under the transfer map is equal to $|G_f|(\sum_{f'\in G/G_f}f')$, where $G_f$ is the stabilizer of $f$. For further background on the transfer map in the context of invariant theory, we refer the reader to \cite{NeuselSmithBook} or \cite{SmithBook}.

When $G$ is a permutation group, a characteristic-free generating set for the image of the transfer map was given by Neusel in \cite{neusel2}. The generating set involves the so-called \textit{special monomials}, which we define below.

\begin{definition}\label{def: special monomials}
    Let $\alpha = (\alpha_1,\ldots,\alpha_n)\in\N^n$ be an exponent vector for a monomial $x^\alpha$ in $\kk[x_1,\ldots,x_n]$. Define $\lambda(\alpha)$ to be the weakly decreasing rearrangement of $\alpha$. We call $x^\alpha$ a \textit{special monomial} if $\lambda(\alpha)$ satisfies
    \begin{enumerate}
        \item[(i)] $\lambda(\alpha)_j \leq n-j$ for each $1\leq j \leq n$, and
        \item[(ii)] $\lambda(\alpha)_j - \lambda(\alpha)_{j+1} \in \{0,1\}$ for each $1\leq j \leq n-1$.
    \end{enumerate}
\end{definition}

\begin{thm}\label{thm: image of transfer generation}\cite[Theorem 1.1]{neusel2}
    Let $G$ be a finite group acting on $\kk[x_1,\ldots,x_n]$ by variable permutation. Then the image of the transfer map is generated by the transfers of special monomials.
\end{thm}

\begin{remark}
    Neusel proves Theorem \ref{thm: image of transfer generation} in the case that $\kk$ is a field of arbitrary characteristic using an induction on dominance order on partitions. This same argument works for any commutative ring $\kk$, hence is stated in this generality here.
\end{remark}

\begin{example}
    The special monomials of $\kk[x_1,x_2,x_3]$ are 
    $$1,\, x_1, \,x_2,\, x_3, \,x_1x_2,\, x_1x_3,\, x_2x_3,\, x_1^2x_2,\, x_1^2x_3, \,x_1x_2^2,\, x_2^2x_3,\, x_1x_3^2, \,x_2x_3^2.$$
\end{example}

When $x^\alpha$ and $x^\beta$ are in the same $G$-orbit, then $\Tr^G(x^\alpha) = \Tr^G(x^\beta)$. Hence when $G = \fS_n$, the image of $\Tr^{\fS_n}$ is generated by transfers of special monomials with weakly decreasing exponent vector, which we sometimes call a \textit{special partition}. When working over $\fS_n$, a special monomial will refer to one with a weakly decreasing exponent vector.

A different generating set for the image of the transfer when $\kk$ has positive characteristic was given by Campbell, Hughes, Shank, and Wehlau in \cite[Theorem 9.18]{CampbellHughesShankWehlau}, using the theory of block bases. This generating set also consists of transfers of certain monomials, and in general is not a minimal generating set. We will see in Section \ref{section: brute force containment} that the minimal generating set of Theorem \ref{thm: main theorem intro} is not given by transfers of monomials.

\subsection{Relating the cofixed space to the transfer map}
\begin{prop}\label{prop: cofixed space is image of transfer}
    Assume that $\kk$ is a commutative ring of characteristic zero. Let $G$ be a finite group acting on $S = \kk[x_1,\ldots,x_n]$ via a permutation representation. Then the cofixed space is isomorphic to the image of the transfer map as an ideal of $S^G$.
\end{prop}

\begin{proof}
    We show that the transfer map $\Tr^G:S_G \to S^G$ is injective. Since $S$ has a $\kk$-basis of monomials, the ring of invariants has a $\kk$-basis of orbit sums of monomials, i.e. a basis of the form
    $$\left\{ m_\lambda := \sum_{x^\alpha \in Gx^\lambda}x^\alpha : \lambda \in \Lambda\right\},$$
    where $\Lambda$ is a complete, irredundant set of $G$-orbit representatives of the set of monomials in $S$, and $Gx^\lambda$ denotes the set of all elements in the $G$-orbit of $x^\lambda$. On the other hand, the cofixed space is a free $\kk$-module with  $\kk$-basis $\{\overline{x^\lambda} : \lambda \in \Lambda\}$, where $\overline{f}$ denotes the image of $f$ in $S_G$. The transfer map sends $\overline{x^\lambda}$ to $|G_{\lambda}|m_\lambda$, where $G_{\lambda}$ is the stabilizer of $x^\lambda$. The image of $\overline{x^\lambda}$ under the transfer map is not zero since $\kk$ has characteristic zero. Hence, the images of the $\overline{x^\lambda}$ are $\kk$-linearly independent in $S^G$, from which it follows that the transfer map $S_G\to S^G$ is injective.
\end{proof}

The transfer map is a very useful tool to study the cofixed space when $\operatorname{char}(\kk)=0$. However, we are also interested in the case that $\kk$ has positive characteristic. In Lemma \ref{lem: base change} below, we show that taking the cofixed space commutes with base change. This allows us to work over $\kk=\Z$ (or $\kk = \Z_{(p)}$) before changing our coefficient ring to, e.g., $\kk = \Z/p\Z$. The next lemma 
was observed and proven by Katth\"an \cite{katthan}. For notes on base change, see \cite[10.14]{StacksProject}. The reader should keep in mind that we intend to apply the following to the case when $k = \Z$ or $k = \Z_{(p)}$, $R = k[e_1,\ldots,e_n]$, $M = k[x_1,\ldots,x_n]$, and $K = \Z/p\Z$.

\begin{lemma}[Base Change Lemma]\label{lem: base change}
    Let $\varphi:k \to R$, $\psi:k\to K$ be homomorphisms of commutative rings. We view $\psi:k\to K$ as the base change map. Let $M$ be a left $kG$-module and an $R$-module such that the actions of $R$ and $kG$ commute. 
    Since $k$ is commutative, view $M$ as an $(kG,k)$-bimodule. Then $$(k\otimes_{kG} M) \otimes_k K\cong k\otimes_{kG} (M\otimes_k K)$$ as $R\otimes_k K$-modules; in other words, $M_G\otimes_k K \cong (M\otimes_k K)_G$ as $R\otimes_k K$-modules.
\end{lemma}

\begin{proof}
    By associativity of tensor product, there is a natural map $a:(k\otimes_{kG}M)\otimes_k K\to k\otimes_{kG}(M\otimes_k K)$ which is well-defined and is an isomorphism of both $k$-modules and $kG$-modules \cite[10.12]{StacksProject}. It remains to check that this map preserves the $(R\otimes_k K)$-module structure on the source and the target. We show how an element $r\otimes y$ of $R\otimes_k K$ acts on basic tensors, where $x_0\in k$, $m_0\in M$, $y_0\in K$:
    \begin{align*}
        (r\otimes y) \cdot ((x_0 \otimes m_0)\otimes y_0)
        &:= r(x_0\otimes m_0) \otimes yy_0 \\
        &= (x_0 \otimes rm_0) \otimes yy_0, 
        \\
        (r\otimes y) \cdot (x_0 \otimes(m_0\otimes y_0)) &:= x_0 \otimes\left((r\otimes y)\cdot (m_0\otimes y_0) \right)\\
        &=x_0 \otimes(rm_0 \otimes yy_0).
    \end{align*}
    Hence the natural map $a$ is also a map of $S\otimes_R R'$-modules, as claimed.
\end{proof}

\begin{remark}
In the setting of Lemma \ref{lem: base change}, we would like to let $R$ be the ring of $G$-invariants inside $\kk[x_1,\ldots,x_n]$ for some group $G$. In general, it is not true that $(R\otimes_\kk K)^G = R^G\otimes_\kk K$ for a base change ring $K$, and Lemma \ref{lem: base change} would not give information on the structure of $(M\otimes_\kk K)_G$ as a module over $K[x_1,\ldots,x_n]^G$. In the case of $G=\fS_n$ with its standard action on the polynomial ring, however, the ring of invariants has the same presentation over any base ring.
\end{remark}

Lemma \ref{lem: base change} shows that we can understand $M=\F_p[x_1,\ldots,x_n]_{\fS_n}$ as a module over $\F_p[e_1,\ldots,e_n]$ by first computing $I^{\fS_n}_\Z\subset \Z[e_1,\ldots,e_n]$ and then taking $I^{\fS_n}_\Z/pI^{\fS_n}_\Z\cong I^{\fS_n}_\Z\otimes_\Z \Z/p\Z$. The ideals $I^{\fS_n}_\Z$ can be very complicated, as the stabilizers of monomials $x^\lambda$ become large. Instead, we work over $\kk=\Z_{(p)}$; in this case, the ideals $I^{\fS_n}_{\Z_{(p)}}$ become much simpler (and more interesting, as Theorem \ref{thm: main theorem intro} and Conjecture \ref{conj: main conj intro} demonstrate),
 while we can deduce information about $M$ in the same way.

\section{Proof of Main Theorem}\label{section: main theorem proof}

We specialize further to the case when $S = \kk[x_1,\ldots,x_n]$ and $G = \fS_n$. Here, the ring of invariants $S^G$ is a polynomial algebra $\kk[e_1,\ldots,e_n]$, where $e_i$ denotes the $i^{\text{th}}$ elementary symmetric polynomial in the variables $x_1,\ldots,x_n$. For each $n=p+i$ with $0\leq i \leq p-1$, we fix the ideal $J_n$ to be
$$J_n = \langle p, \,\, e_1e_p - e_{p+1},\,\,e_2e_p-e_{p+2},\,\,\ldots,\,\, e_ie_p-e_{p+i},\,\, e_{i+1},\,\,e_{i+2},\,\,\ldots,\,\, e_{p-1}\rangle\subset \Z_{(p)}[e_1,\ldots,e_n].$$
To prove Theorem \ref{thm: main theorem intro}, we will first show that there is a containment $J_n\subseteq I^{\fS_n}$ and that the theorem holds when $n=p$. We will then use a stability present in the $p$-Sylow subgroups of $\fS_n$, along with a result of Shank and Wehlau \cite{ShankWehlau} to deduce the height of $I^{\fS_n}$, which will be key in proving the theorem. Throughout this section, set $\Tr = \Tr^{\fS_n} = \Tr^{\fS_n}_{\Z_{(p)}}$, unless otherwise specified.

\begin{remark}
One can readily see via the Jacobi--Trudi identities that each generator $e_pe_j-e_{p+j}$ of the ideal $J_n$ is the skew-Schur polynomial associated to the two-column ribbon shape, with the first column having $j$ boxes and the second column having $p$ boxes. For example, the generator $e_2e_3-e_5$ is the Schur polynomial for the skew shape
$
\ytableausetup{smalltableaux}
\begin{ytableau}
\none & \\ 
\none & \\
 & \\
 & \none
\end{ytableau}.$
\end{remark}

\subsection{Containment of $J_n$ in $\imtransfer{\fS_n}$}\label{section: brute force containment}

\begin{prop}\label{prop: I contained in Im(Tr)}
Given $n = p+i$ with $1 \leq i \leq p-1$, the ideal $J_n$ of Theorem \ref{thm: main theorem intro} is contained in $\imtransfer{\fS_n}$.
\end{prop}

\begin{proof}
    We separate the generators of $J_n$ into three types and show that each type of generator lies in $\imtransfer{\fS_n}$:
    \begin{enumerate}
        \item[(i)] $p$,
        \item[(ii)] $e_j$ for $i+1 \leq j \leq p-1$, and
        \item[(iii)] $e_je_p - e_{p+j}$ for $1\leq j \leq i$.
    \end{enumerate}
    It is immediate that $p\in \imtransfer{\fS_n}$, since
        $$\Tr\left(\frac{p}{n!}\cdot 1 \right) = \frac{p}{n!} \Tr(1) = p$$
    and $\frac{n!}{p}$ is invertible in $\Z_{(p)}$ for $n < 2p \leq p^2$. Writing $\Tr(x_1\cdots x_j) =j!(p+i-j)!e_j$, it follows that $e_j\in I^{\fS_n}$ if and only if $j\leq p-1$ and $p+i-j\leq p-1$. These are exactly the conditions needed for $e_j$ to lie in $J_n$.

    It is more complicated to show that $e_je_p-e_{p+j}$ is in the image of the transfer when $1\leq j \leq i$. For any partition $\lambda=(\lambda_1,\ldots,\lambda_n)$, let $a_\lambda$ be the order of the stabilizer of $\lambda$ in $\fS_n$, where $\fS_n$ acts by coordinate permutation. As we have seen, we can write $\Tr(x^\lambda) = a_\lambda m_\lambda$, where $m_\lambda$ is the monomial symmetric polynomial associated to $\lambda$, that is, the sum of all distinct elements in the $\fS_n$-orbit of $\lambda$. The $m_\lambda$ form a $\kk$-basis of $\kk[x_1,\ldots,x_n]^{\fS_n}$, for any $\kk$, as $\lambda$ runs over all weakly decreasing exponent vectors. Hence, for some $b_\lambda \in \mathbb{Q}$ we can write
        $$e_je_p - e_{p+j} = \sum_{\lambda} b_\lambda m_\lambda = \sum_\lambda \frac{b_\lambda}{a_\lambda} \Tr_{\mathbb{Q}}(x^\lambda).$$

    We will compute the coefficients $b_\lambda$ explicitly, then show that $\frac{b_\lambda}{a_\lambda}\in\Z_{(p)}$, that is, that $\frac{b_\lambda}{a_\lambda}$ has no powers of $p$ in the denominator when written in lowest terms. From this it follows that $e_je_p - e_{p+j}\in\imtransfer{\fS_n}$. When expanding $e_je_p-e_{p+j}$ in terms of monomial symmetric polynomials $m_\lambda$, the partitions $\lambda$ appearing with nonzero coefficient are of the form
        $$\lambda^{(k)} = (2^k, 1^{p+j-2k}, 0^{i-j+k})$$ 
    for $0\leq k \leq j$. We will consider the case $k=0$ separately later in the proof. For now, assume that $k\geq 1$. Then the coefficient $b_{\lambda^{(k)}}$ on $m_{\lambda^{(k)}}$ in $e_je_p$ is equal to $\binom{p+j-2k}{j-k}$. To see this, we can count how many times the monomial $x_1^2 \cdots x_k^2 x_{k+1} \cdots x_{p+j-k}$ appears in $e_je_p$. In order to obtain such a monomial, one must choose a term from $e_j$ divisible by $x_1\cdots x_k$ and a term from $e_p$ divisible by $x_1\cdots x_k$. It remains to choose the variables $x_{k+1}, \ldots, x_{k+p+j-2k}$ (which is a set of size $p+j-2k$), and $j-k$ of these must come from the $e_j$ term. Since $k\geq 1$, this is also the coefficient of $m_{\lambda^{(k)}}$ in $e_je_p-e_{p+j}$. 

    Since $a_\lambda = k!(p+j-2k)!(i-j+k)!$, we can write
    \begin{align*}
        e_je_p - e_{p+j} &= \frac{b_{(1^{p+j})}}{a_{(1^{p+j})}}\Tr_{\mathbb{Q}}(x^{(1^{p+j})}) + \sum_{k=1}^j \frac{\binom{p+j-2k}{j-k}}{k!(p+j-2k)!(i-j+k)!}\Tr_{\mathbb{Q}}(x^{\lambda^{(k)}}) \\
        &= \frac{b_{(1^{p+j})}}{a_{(1^{p+j})}}\Tr_{\mathbb{Q}}(x^{(1^{p+j})}) + \sum_{k=1}^j \frac{1}{(j-k)!(p-k)!k!(i-j+k)!}\Tr_{\mathbb{Q}}(x^{\lambda^{(k)}}).
    \end{align*}
    Observe that in each coefficient on $\Tr_{\mathbb{Q}}(x^{\lambda^{(k)}})$, for $k \geq 1$, there are no factors of $p$ in the denominator. Hence these coefficients are elements of $\Z_{(p)}$. It remains to determine the coefficient of $\Tr_{\mathbb{Q}}(x^{(1^{p+j})})$. Set $\lambda = (1^{p+j})$. We have $b_\lambda = \binom{p+j}{j}-1$, since $e_je_p$ gives $\binom{p+j}{j}$ terms $x_1\cdots x_{p+j}$, and $e_{p+j}$ has exactly one term of this form. Since $\lambda$ consists of $p+j$ 1's and $p+i-(p+j)$ 0's, we have $a_\lambda = (p+j)!(i-j)!$. Hence,
        $$ \frac{b_\lambda}{a_\lambda} = \frac{\binom{p+j}{j} - 1 }{(p+j)!(i-j)!} = \frac{(p+j)!-p!j!}{p!j!(p+j)!(i-j)!}.$$ 
    Now we can rewrite $(p+j)! = (p+j)(p+j-1)\cdots p!$ and cancel a copy of $p!$, giving
        $$\frac{b_\lambda}{a_\lambda} = \frac{(p+j)(p+j-1)\cdots (p+1) - j!}{j!(p+j)(p+j-1)\cdots (p)(p-1)!(i-j)!}.$$
    Note that $(p+j)(p+j-1)\cdots (p+1) \equiv j! \mod p$, and it is strictly larger than $p$. Hence we can write $(p+j)(p+j-1)\cdots (p+1) = j! + \ell p$ for some $\ell>0$. Moreover, there is only one factor of $p$ in the denominator. Hence we have
        $$\frac{b_\lambda}{a_\lambda} = \frac{j!+\ell p - j!}{j!(p+j)\cdots p(p-1)!(i-j)!} = \frac{\ell}{j!(p+j)\cdots (p+1)(p-1)!(i-j)!}\in\Z_{(p)}.$$
\end{proof}

\subsection{A base case}

\begin{prop}\label{prop: n=p}
    When $n = p$, the image of the transfer map is equal to the ideal $J_p = \langle p, e_1,\ldots,e_{p-1}\rangle$ of $\Z_{(p)}[e_1,\ldots,e_p]$.
\end{prop}

Before proceeding with the proof of Proposition \ref{prop: n=p}, we recall the definition of dominance order on partitions. The technique of induction on (degree and) dominance order is similar to what is used in Neusel's proof of \cite[Theorem 1.1]{neusel2}.

\begin{definition}
    Fix an integer $m\geq1$ and let $\lambda =(\lambda_1,\ldots,\lambda_n)$, $\mu = (\mu_1,\ldots,\mu_n)$ be two integer partitions of $m$, where the tuples $\lambda$ and $\mu$ have weakly decreasing nonnegative coordinates. We say that $\lambda$ \textit{dominates} $\mu$, or $\lambda\geqdom \mu$, if for each $1\leq j \leq n$, one has
    $\lambda_1 + \cdots + \lambda_j \geq \mu_1 + \cdots + \mu_j.$
\end{definition}

\begin{proof}[Proof of Proposition \ref{prop: n=p}]
    By Proposition \ref{prop: I contained in Im(Tr)}, there is a containment $J_p\subseteq I^{\fS_p}$. To show that the generators of $J_p$ also generate the image of the transfer, it is enough to show that the transfer of any special monomial (see Definition \ref{def: special monomials}) lies in $J_p$, by Theorem \ref{thm: image of transfer generation}. We will use induction on degree and dominance order. Let $x^\lambda$ be a special monomial such that $\lambda_1\geq 2$. Set $k:=\max\{j : \lambda_j\neq 0\}$. Since $\lambda$ is special and that $\lambda_1\geq 2$, we have $k\in\{2,\ldots,p-1\}$ with $\lambda_k = 1$. Define a new partition
        $$\widetilde\lambda :=(\lambda_1 - 1,\lambda_2-1,\ldots, \lambda_{k-1}-1,0,\ldots,0)$$
    obtained from $\lambda$ by subtracting 1 from all nonzero parts.
    Writing $e_k\, x^{\widetilde\lambda} = \sum_{1 \leq i_1 < \cdots < i_k \leq p}(x_{i_1}\cdots x_{i_k} x^{\widetilde\lambda})$, we can apply the transfer map to both sides and use its $\Z_{(p)}[e_1,\ldots,e_p]$-linearity:
    \begin{equation}\label{eq: n=p reduced lambda}
        e_k \Tr(x^{\widetilde\lambda}) = \sum_{1\leq i_1 < \cdots < i_k\leq p}\Tr(x_{i_1}\cdots x_{i_k}x^{\widetilde\lambda}).
    \end{equation}
    A generic term of the right-hand side is of the form $\Tr(x^\alpha)$, where $\alpha=(\alpha_1,\ldots,\alpha_p)$ is obtained from $\widetilde\lambda$ by adding 1 to exactly $k$ of its entries. The term $\Tr(x^\lambda)$ appears only when a 1 is added to every nonzero entry of $\widetilde\lambda$. Each nonzero entry of $\widetilde\lambda$ corresponds to an entry of $\lambda$ that is larger than 2. If $\widetilde\lambda$ has $\ell$ nonzero entries, then $1\leq \ell < k$ since $\lambda_k=1$. Hence the coefficient of $\Tr(x^\lambda)$ in (\ref{eq: n=p reduced lambda}) is $\binom{p-\ell}{k-\ell}$, which is invertible in $\Z_{(p)}$. Now given any term $\Tr(x^\alpha)$ in (\ref{eq: n=p reduced lambda}), let $\mu$ be the weakly decreasing rearrangement of $\alpha$. By construction, $\mu_j\leq \lambda_j$ for each $j\leq k$, with equality occurring if and only if $\mu=\lambda$. For $j>k$ we have 
        $$\mu_1+\cdots + \mu_j \leq \deg(x^\lambda) = \lambda_1+\cdots + \lambda_k = \lambda_1 + \cdots + \lambda_j.$$
    Hence, $\mu\leqdom \lambda$. It follows that we can write
        $$\Tr(x^\lambda) = \frac{1}{\binom{p-\ell}{k-\ell}}\left(e_k\Tr(x^{\widetilde\lambda}) + \sum_{\mu \lessdom \lambda} c_\mu \Tr(x^\mu)\right),$$
    so by induction $\Tr(x^\lambda)\in J_p$.
\end{proof}

\subsection{Using the $p$-Sylow subgroups of $\fS_n$}\label{subsection: p-Sylow}

To prove Theorem \ref{thm: main theorem intro}, we will make use of the fact that the $p$-Sylow subgroups of $\fS_n$ are isomorphic when $n$ lies in the range $\{kp, kp+1,\ldots, (k+1)p - 1\}$ for some $k$.

\begin{lemma}\label{lem: p-Sylows of Sn jump}
    Let $n,m$ be positive integers such that $kp \leq n < m < (k+1)p$ for some positive integer $k$. Then the $p$-Sylow subgroups of $\fS_n$ and $\fS_m$ are isomorphic. Moreover, any given $p$-Sylow subgroup of $\fS_n$ can be embedded into a $p$-Sylow subgroup of $\fS_m$ via the natural inclusion $\fS_n\into \fS_m$.
\end{lemma}

\begin{proof}
    Let $P_n$ be any $p$-Sylow subgroup of $\fS_n$. Then there is a subgroup $P_m$ of $\fS_m$ which is obtained from $P_n$ by applying the map 
    \begin{align*}
        \fS_n &\longrightarrow \fS_m, \\ 
        [w_1, \ldots, w_n] &\mapsto [w_1,\ldots,w_n, n+1, \ldots, m]
    \end{align*}
    where the permutations above are expressed in one-line notation. Then $P_n\cong P_m$ as groups. Because $n,m\in \{kp, \ldots, (k+1)p-1\}$, the orders of $\fS_n$ and $\fS_m$ share the same number of factors of $p$. Hence, the $p$-Sylow subgroups of $\fS_n$ and $\fS_m$ have the same order. Since we have exhibited a subgroup of $\fS_m$ of the correct order, it must be that all $p$-Sylow subgroups of $\fS_m$ are isomorphic to $P_m$.
\end{proof}

From now on, if $n,m$ satisfy $kp\leq n < m < (k+1)p$ for some $k$, we choose $p$-Sylow subgroups $P_n, P_m$ of $\fS_n, \fS_m$, respectively, so that $P_m$ is the image of $P_n$ under the natural inclusion $\iota:\fS_n \into \fS_m$, as in the proof of Lemma \ref{lem: p-Sylows of Sn jump}. When $P_n, P_m$ are chosen in this way, the variables $x_{n+1},\ldots,x_m \in \kk[x_1,\ldots,x_m]$ are all invariant under the action of $P_m$. The simple corollary below follows.

\begin{cor}\label{cor: invariant ring of p-Sylow}
    With $n, m, P_n, P_m$ as above, $\kk[x_1,\ldots,x_m]^{P_m} = \kk[x_1,\ldots,x_n]^{P_n}\otimes_\kk \kk[x_{n+1},\ldots,x_m]$, and there is a natural inclusion $\kk[x_1,\ldots,x_n]^{P_n}\into \kk[x_1,\ldots,x_m]^{P_m}$.
\end{cor}

We may use Lemma \ref{lem: p-Sylows of Sn jump} and Corollary \ref{cor: invariant ring of p-Sylow} to deduce a relationship between the transfer ideals $\imtransfer{P_n}$ and $\imtransfer{P_{n+1}}$, as long as $n+1$ is not a multiple of $p$.

\begin{lemma}\label{lem: p-Sylow extension ideal}
    Let $n \in\{kp, \ldots, (k+1)p - 2\}$ and let $P_n, P_{n+1}$ be $p$-Sylow subgroups of $\fS_n, \fS_{n+1}$ satisfying $\iota(P_n) = P_{n+1}$. Let $R_n$ and $R_{n+1}$ denote the $P_n$- and $P_{n+1}$-invariant subrings of $\kk[x_1,\ldots,x_n]$ and $\kk[x_1,\ldots,x_{n+1}]$, respectively. Then, $\imtransfer{P_{n+1}}$ is the extension ideal $R_{n+1}\cdot\imtransfer{P_n}$.
\end{lemma}

\begin{proof}
    The ideal $\imtransfer{P_{n+1}}$ is generated by the transfers of all monomials $x^\alpha$ in $\kk[x_1,\ldots,x_{n+1}]$. Given such a monomial, either $x_{n+1}$ divides $x^\alpha$ or it does not. If $x_{n+1}$ does not divide $x^\alpha$, then
        $$\transfer{P_{n+1}}(x^\alpha) = \transfer{P_n}(x^\alpha) \in R_{n+1}\cdot\imtransfer{P_n}.$$
    If $x_{n+1}$ divides $x^\alpha$, we can factor $x^\alpha$ as $x^{\widetilde\alpha}\cdot x_{n+1}^b$ with $\gcd(x^{\widetilde\alpha}, x_{n+1}^b) = 1$. Using the fact that $x_{n+1}$ is fixed by $P_{n+1}$ and the action of $P_{n+1}$ is multiplicative, we have
    \begin{align*}
        \transfer{P_{n+1}}(x^\alpha) &= \sum_{w \in P_{n+1}} (wx^{\widetilde\alpha})(wx_{n+1}^b) \\
        &=\sum_{w\in P_{n+1}} (wx^{\widetilde\alpha})x_{n+1}^b \\
        &= x_{n+1}^b \sum_{w\in \iota(P_n)}wx^{\widetilde\alpha}\\
        &= x_{n+1}^b \transfer{P_n}(x^{\widetilde\alpha}) \quad \in R_{n+1}\cdot\imtransfer{P_n}.
    \end{align*}
    We have shown the containment $\imtransfer{P_{n+1}}\subseteq R_{n+1}\cdot\imtransfer{P_n}$. To show the reverse containment, take $f \in R_{n+1}\cdot\imtransfer{P_n}$ and write $f = \sum_j r_j \transfer{P_n}(f_j)$ for some $r_j \in R_{n+1}$ and $f_j \in \kk[x_1,\ldots,x_n]$. Since the $f_j$ use only the variables $x_1,\ldots,x_n$, we have $\transfer{P_n}(f_j) = \transfer{P_{n+1}}(f_j)$. Since $\transfer{P_{n+1}}$ is a map of $R_{n+1}$-modules and each $r_j \in R_{n+1}$, it follows that $f = \transfer{P_{n+1}}\left(\sum_j r_j\cdot f_j\right)\in \imtransfer{P_{n+1}}$.
\end{proof}

\begin{cor}\label{cor: height of p-sylow}
    With $\imtransfer{P_n}\subset R_n$ and $\imtransfer{P_{n+1}}\subset R_{n+1}$ as before,
        $\hgt(\imtransfer{P_n}) = \hgt(\imtransfer{P_{n+1}}).$
\end{cor}

\begin{proof}
    The extension of rings $R_n \into R_{n+1}$ is flat, hence the going down theorem holds. Moreover, the induced map $\Spec{R_{n+1}} \to \Spec{R_n}$ is surjective. By \cite[Theorem 19 (3)]{matsumura}, the heights of $\imtransfer{P_n}$ and its extension ideal $\imtransfer{P_{n+1}} = R_{n+1}\cdot\imtransfer{P_n}$ are equal.
\end{proof}

Now that we have related $\imtransfer{P_n}$ and $\imtransfer{P_{n+1}}$, we can compare the heights of $\imtransfer{P_n}$ and $\imtransfer{\fS_n}$. To do so, we appeal to a corollary found in \cite{ShankWehlau}, restated slightly to accommodate the case that $\kk$ is not a field.

\begin{cor}\label{cor: height of p Sylow equals height of group}\cite[Corollary 5.2]{ShankWehlau}
    Assume that $\kk$ is a commutative ring in which $|\fS_n:H|$ is invertible, where $H$ is a subgroup of $\fS_n$ acting on $S:=\kk[x_1\ldots,x_n]$. Then $\hgt(\imtransfer{\fS_n}) = \hgt(\imtransfer{H}).$
\end{cor}

\begin{proof}
    Under the assumption that $|\fS_n:H|\in \kk^\times$, it follows from \cite[Proposition 5.1]{ShankWehlau}, that $I^H$ lies over $I^{\fS_n}$ in the integral ring extension $S^{\fS_n} \into S^H$. Since both $S^{\fS_n}$ and $S^H$ are integral domains and $S^{\fS_n}$ is integrally closed\footnote{When $\kk=\Z$ or $\kk=\Z_{(p)}$, the ring of invariants $S^{\fS_n}$ is a polynomial ring over a UFD, hence is a UFD itself. For $\kk$ a field, it is well-known that any ring of invariants of a finite group is integrally closed; see \S1.7 of \cite{NeuselSmithBook}.}, the going down theorem holds for this extension of rings. By \cite[Theorem 20 (3)]{matsumura}, we have $\hgt(I^{\fS_n}) = \hgt(I^H)$.
\end{proof}

\subsection{Proof of Theorem \ref{thm: main theorem intro}}\label{subsection: main proof}

It remains to put together the content in the previous sections to prove our main theorem. We apply Corollary \ref{cor: height of p Sylow equals height of group} in the case that $\kk = \Z_{(p)}$ and $H = P_n$.

\begin{proof}[Proof of Theorem \ref{thm: main theorem intro}]
Combining Corollary \ref{cor: height of p-sylow} with Corollary \ref{cor: height of p Sylow equals height of group}, we conclude that for any $~n,m \in \{kp, kp+1,\ldots,(k+1)p-1\}$ for some $k\geq 1$, we have $$\hgt(\imtransfer{\fS_n}) = \hgt(I^{P_n}) = \hgt(I^{P_m}) = \hgt(\imtransfer{\fS_m}).$$ When $n \in \{p,\ldots, 2p-1\}$, Proposition \ref{prop: n=p} shows that $\hgt(\imtransfer{\fS_n}) = \hgt(\imtransfer{\fS_p}) = p$. Since $\hgt(\imtransfer{\fS_n}) = p$, we can find a prime $\qq_n$ of height $p$ in $\Z_{(p)}[e_1,\ldots,e_n]$ such that $\imtransfer{\fS_n} \subseteq \qq_n$. On the other hand, Proposition \ref{prop: I contained in Im(Tr)} shows that $J_n \subseteq \imtransfer{\fS_{n}}$. But $J_n$ is itself a prime ideal of height $p$, hence we have a containment
    $$J_n \subseteq \imtransfer{\fS_n} \subseteq \qq_n,$$
with $J_n$ and $\qq_n$ both primes of height $p$. From this we conclude that $J_n  = I^{\fS_n}= \qq_n$, proving the theorem.
\end{proof}

\section{Applications}\label{section: applications}

\subsection{The cofixed space with $\F_p$ coefficients}

By Lemma \ref{lem: base change}, we can find the structure of $\F_p[x_1,\ldots,x_n]_{\fS_n}$ as an $\F_p[e_1,\ldots,e_n]$-module by studying the module $I^{\fS_n}_{\Z_{(p)}} \otimes_\Z \Z/p\Z$, on which $\F_p[e_1,\ldots,e_n]$ acts by multiplication on the left tensor factor. 

\begin{thm}\label{thm: F_p structure}
Let $n = p+i$ with $0\leq i\leq p-1$. Then the cofixed space $\F_p[x_1,\ldots,x_n]_{\fS_n}$ has the following direct sum decomposition as a module over $\F_p[e_1,\ldots,e_n]$:
    \begin{equation}\label{eq: Fp cofixed decomp}
    \F_p[x_1,\ldots,x_n]_{\fS_n} \cong \frac{\F_p[e_1,\ldots,e_n]}{\widetilde{J}_n} \oplus \widetilde{J}_n
    \end{equation}
    where $\widetilde{J}_n$ is the ideal $\langle e_1e_p -e_{p+1},\, e_2e_p-e_{p+2},\,\ldots,e_ie_p-e_{p+i},\,e_{i+1},\, e_{i+2},\,\ldots,e_{p-1}\rangle$ of $\F_p[e_1,\ldots,e_n]$.
\end{thm}

\begin{proof}
    Let $R = \F_p[e_1,\ldots,e_n]$. We first define a map of graded $R$-modules by
    \begin{align*}
        \varphi: I^{\fS_n} \otimes_\Z \Z/p\Z &\longrightarrow \widetilde{J}_n, \\
        f\otimes a &\mapsto af.
    \end{align*}
    This is a surjective map since any generator $g$ of $\tilde{J}_n$ has $g\otimes 1$ as a preimage. Any basic tensor in $\ker(\varphi)$ must have a representative $f\otimes a$ satisfying either that $a=0$ in $\Z/p\Z$ or that $f$ is a multiple of $p$ in $I^{\fS_n}$. Hence, any non-zero basic tensor in $\ker(\varphi)$ must be an $R$-multiple of $p\otimes 1$. On the other hand, any element of $I^{\fS_n}\otimes_\Z \Z/p\Z$ is a basic tensor: If $x=\sum r_j(f_j \otimes a_j)$ is an element of $I^{\fS_n}\otimes_\Z \Z/p\Z$, where $r_j\in R, f_j\in I^{\fS_n}$, and $a_j\in\Z$ is a representative of $a_j+p\Z$, then $x$ can be rewritten
    $$x=\sum r_j(f_j \otimes a_j) = \sum (r_jf_j\otimes a_j) = \sum (a_jr_jf_j\otimes 1) = \left(\sum a_jr_jf_j\right)\otimes1.$$
    Hence, $\ker(\varphi)$ is a cyclic $R$-module generated by $p\otimes 1$. Moreover, $r(p\otimes 1)=0$ for some $r\in R$ if and only if $r\in\tilde{J}_n$, so $\ker(\varphi)\cong R/\tilde{J}_n$ as graded $R$-modules.
    
    We now construct a right inverse for $\varphi$. Let the free module $R^{p-1}$ have $R$-basis $b_1,\ldots,b_{p-1}$. Define a map
    \begin{align*}
        \hat\psi: R^{p-1} &\longrightarrow I^{\fS_n}\otimes_\Z \Z/p\Z \\
        b_j &\mapsto \begin{cases}
            (e_je_p - e_{p+j})\otimes 1 & \text{if }1\leq j \leq i \\
            e_j\otimes 1 & \text{if }i+1\leq j \leq p-1
        \end{cases}
    \end{align*}
    and extend $R$-linearly. Label the minimal generators of $\tilde{J}_n$ by $f_1,\ldots,f_{p-1}$, where $f_j$ has degree $j$ mod $p$. Then the map $\hat\psi$ satisfies
    $$f_k\hat\psi(b_j) - f_j\hat\psi(b_k) = 0$$
    for each $1\leq j,k\leq p-1$. Since the generators of $\tilde{J}_n$ form an $R$-regular sequence, the map $\hat\psi$ descends to a map $\psi:\tilde{J}_n\to I^{\fS_n}\otimes_\Z \Z/p\Z$. Moreover, $\psi$ is a right inverse for $\varphi$, and therefore $I^{\fS_n}\otimes_\Z \Z/p\Z$ splits as a direct sum of $R$-modules
    $$I^{\fS_n}\otimes_\Z \Z/p\Z \cong \ker(\varphi)\oplus \tilde{J}_n \cong R/\tilde{J}_n\oplus \tilde{J}_n.$$
\end{proof}

\subsection{The image of the transfer map with $\F_p$ coefficients}

\begin{cor}\label{cor: Fp transfer}
    Let $n \in \{p,p+1,\ldots,2p-1\}$. Then the image of the transfer map $\Tr^{\fS_n}_{\F_p}:\F_p[x_1,\ldots,x_n]\to \F_p[e_1,\ldots,e_n]$ is equal to the ideal
    $$\tilde{J}_n = \langle e_1e_p -e_{p+1},\,\, e_2e_p-e_{p+2},\,\ldots,\,e_ie_p-e_{p+i},\,\,e_{i+1},\,\, e_{i+2},\,\ldots,\,e_{p-1}\rangle \subset \F_p[e_1,\ldots,e_n].$$
\end{cor}

\begin{proof}
    When $G$ is a permutation group, one can verify using a monomial basis for the polynomial ring that the square
    \begin{center}
    \begin{tikzcd}
    \Z_{(p)}[x_1,\ldots,x_n] \arrow[d, two heads] \arrow[r, "\Tr^G_{\Z_{(p)}}"] & \Z_{(p)}[x_1,\ldots,x_n]^G \arrow[d, two heads] \\
    \F_p[x_1,\ldots,x_n] \arrow[r, "\Tr^G_{\F_p}"] & \F_p[x_1,\ldots,x_n]^G
    \end{tikzcd}
    \end{center}
    commutes. Since the map $\Z_{(p)}[x_1,\ldots,x_n] \to \F_p[x_1,\ldots,x_n]$ is surjective, the image of $\Tr^{\fS_n}_{\F_p}$ can found by taking $I^{\fS_n}_{\Z_{(p)}}$ of Theorem \ref{thm: main theorem intro} and applying the natural surjection to $\F_p[e_1,\ldots,e_n]$. This gives exactly the ideal $\tilde{J}_n$.
\end{proof}

Corollary \ref{cor: Fp transfer} is consistent with several other results on the image of the transfer map for modular representations $G\into \GL_n(\kk)$, where $\kk$ is a field of characteristic $p>0$. Campbell, Hughes, Shank, and Wehlau remark that the generating set for $\tilde{J}_p$ follows from \cite[Theorem 9.18]{CampbellHughesShankWehlau}, where they give a block basis for the image of $\Tr^{\fS_n}_\kk$, which in general is a redundant generating set. Shank and Wehlau showed that $I^G_{\kk}$ is radical when $G$ acts by permutations \cite[Theorem 6.1]{ShankWehlau}. 

In \cite[Corollary 2.5]{Neusel1}, Neusel showed that if $G$ has a cyclic $p$-Sylow subgroup which acts by permutations, then $I^G_\kk$ has height at most $n-k$, where $k$ is the number of orbits of $G$ acting on $x_1,\ldots,x_n$. When $p \leq n < 2p$, the $p$-Sylow subgroups of $\fS_n$ are cyclic of order $p$. One can choose a $p$-Sylow subgroup of $\fS_n$ which cyclically permutes $x_1,\ldots,x_p$ and fixes each of $x_{p+1},\ldots,x_n$. Such a subgroup has $1 + (n-p)$ orbits, so $I^{\fS_n}_\kk$ must have height at most $n - (1+n-p) = p-1$. When $\kk = \F_p$, Corollary \ref{cor: Fp transfer} shows that Neusel's bound is sharp, since each $\tilde{J}_n$ is prime of height $p-1$.

\section{A Conjecture for Larger $n$}\label{section: conjecture}

Theorem \ref{thm: main theorem intro} suggests a natural question regarding the cofixed spaces $\Z_{(p)}[x_1,\ldots,x_n]_{\fS_n}\cong I^{\fS_n}$ for $n \geq 2p$. Because the generators for $I^{\fS_n}$ form a regular sequence when $n=p,p+1,\ldots,2p-1$, the resolution of $I^{\fS_n}$ over $\Z_{(p)}[e_1,\ldots,e_n]$ is a Koszul complex. Moreover, since the generators always have degrees $0,1,\ldots,p-1 \mod p$, taking all graded Betti numbers $\beta_{i,j}$ for fixed $i$ gives the same multiset mod $p$. Such stability in the $S^{\fS_n}$-module structure is trivially true for $n < p$, since here $S_{\fS_n}$ is always a free $S^{\fS_n}$-module of rank 1. One may ask to what extent this phenomenon persists. The $p$-Sylow subgroups of $\fS_n$ follow a similar stability pattern for all $n$, and this structure was critical to the proof of Theorem \ref{thm: main theorem intro}. It is conceivable that this phenomenon is present for all $n$.

To state our conjecture, we make a brief definition. Given a $\Z$-graded module $M$ over a $\Z$-graded ring $R$ and some $i\geq 0$, let $A^R_i(M)$ denote the multiset of integers in which $j\in\Z$ occurs exactly $\beta_{i,j}^R(M)$ times. In other words, the multiset $A_i^R(M)$ records the degree shifts appearing in the $i^{\text{th}}$ homological degree in a graded minimal free resolution of $M$ over $R$.

\begin{conj}\label{conj: betti number}
    For each $r\geq 1$, let $R_r:=\kk[e_1,\ldots,e_r]$, $M_r:=\kk[x_1,\ldots,x_r]_{\fS_r}$. Let $m,n$ satisfy $~kp \leq ~m,n < (k+1)p$ for some integer $k$, and assume $\kk = \F_p$ or $\kk = \Z_{(p)}$. Then for each $i\geq 0$, the multisets $A^{R_n}_i(M_n)$ and $A^{R_m}_i(M_m)$ are equal after taking all elements mod $p$.
\end{conj}

We present some evidence for Conjecture \ref{conj: betti number}. All computations were done in \texttt{Macaulay2}, and the data listed below uses $\kk = \F_p$. We first present data for $n$ in the range $\{2p, 2p+1,\ldots,3p-1\}$ for $p=2$ and $p=3$.

When $p=2$, we consider $n \in \{4,5\}$. Below we list the elements of $A_i(M_4)$ and $A_i(M_5)$; the entries in the $i^{\text{th}}$ row and the second column give the degree twists which show up in a minimal free resolution of $M_4, M_5$ over $R_4, R_5$, respectively. We draw the reader's attention to the fact that the rightmost columns of both tables in Figure \ref{fig: p2n45} are identical. When $p=3$, data is available for $n\in\{6,7\}$. The rightmost columns of the tables in Figure \ref{fig: p3n67} are again identical.

\begin{figure}[h!]
\begin{minipage}{0.42\textwidth}
\begin{center}
\begin{tabular}{|Sc|Sl|Sl|}
    \hline
    $i$ & \multicolumn{1}{c|}{$A_i(M_4)$} & \multicolumn{1}{c|}{$A_i(M_4) \mod 4$}\\
    \hline
    $0$ & $0, 1, 2, 3, 6$ & $0, 1, 2, 3, 2$ \\
    \hline
    $1$ & $1, 2, 3, 3, 4, 5, 6$ &  $1, 2, 3, 3, 0, 1, 2$\\
    \hline
    $2$ & $3,4,5,6$ & $3,0,1,2$ \\
    \hline
    $3$ & $6$ & $2$ \\ \hline
\end{tabular}
\end{center}
\end{minipage}
\begin{minipage}{0.49\textwidth}
\begin{center}
\begin{tabular}{|Sc|Sl|Sl|}
    \hline
    $i$ & $A_i(M_5)$ & $A_i(M_5) \mod 4$\\
    \hline
    $0$ & $0,5,2,3,10$ & $0, 1, 2, 3, 2$ \\
    \hline
    $1$ & $5,2,3,7,8,5,10$  & $1, 2, 3, 3, 0, 1, 2$\\
    \hline
    $2$ & $7,8,5,10$ &  $3,0,1,2$ \\
    \hline
    $3$ & $10$ & $2$ \\ \hline
\end{tabular}
\end{center}
\end{minipage}
\caption{Resolution data for $p=2$, $n=4$ (first table) and $p=2$, $n=5$ (second table). }\label{fig: p2n45}
\end{figure}

\begin{figure}[h!]
\begin{minipage}{0.48\textwidth}
\begin{tabular}{|Sc|Sl|Sl|}
    \hline
    $i$ & \multicolumn{1}{c|}{$A_i(M_6)$} & \multicolumn{1}{c|}{$A_i(M_6)\mod6$} \\
    \hline
    $0$ & $0,1,2,4,5,6,8,9,10$ & $0,1,2,4,5,0,2,3,4$ \\
    \hline
    $1$ & $1,2,3,4,5,5,6,6,6,$ & $1,2,3,4,5,5,0,0,0,$ \\
    \, & $7,8,9,9,10,10,11,13,$ & $1,2,3,3,4,4,5,1,$\\
     \, & $14$&  $2$\\
    \hline
    $2$ & $3,5,6,6,7,7,8,9,10,$ & $3,5,0,0,1,1,2,3,4$ \\
    \, &
    $10,11,11,13,14,15$ & $4,5,5,1,2,3$ \\
    \hline
    $3$ & $7,8,10,11,12,15$ & $1,2,4,5,0,3$ \\ \hline
    $4$ & $12$ & $0$ \\ \hline
\end{tabular}
\end{minipage}
\begin{minipage}{0.48\textwidth}
\begin{tabular}{|Sc|Sl|Sl|}
    \hline
    $i$ & \multicolumn{1}{c|}{$A_i(M_7)$} & \multicolumn{1}{c|}{$A_i(M_7)\mod 6$}\\
    \hline
    $0$ & $0,7,2,4,5,12,14,9,10$ & $0,1,2,4,5,0,2,3,4$ \\
    \hline
    $1$ & $7,2,9,4,5,11,6,12,12,$ &  $1,2,3,4,5,5,0,0,0,$\\
    \, & $7,14,9,10,16,17,19,$ &  $1,2,3,3,4,4,5,1,$\\
    \, & $14$ &  $2$\\
    \hline
    $2$ & $9,11,6,12,7,13,14,9,16$ & $3,5,0,0,1,1,2,3,4$\\
    \, & $16,11,17,19,14,21$ & $4,5,5,1,2,3$ \\ 
    \hline
    $3$ & $13,14,16,11,18,21$ & $1,2,4,5,0,3$ \\
    \hline
    $4$ & $18$ & $0$ \\ \hline
\end{tabular}
\end{minipage}
\caption{Resolution data for $p=3$, $n=6$ (first table) and $p=3$, $n=7$ (second table). }\label{fig: p3n67}
\end{figure}

This data may suggest that something stronger than Conjecture \ref{conj: betti number} holds, namely that $A_i^{R_n}(M_n)$ and $A_i^{R_m}(M_m)$ are the same mod $2p$ when $2p \leq n,m < 3p$. So far, no counterexample to this claim has been found. One may hope that more generally, if $kp \leq m,n < (k+1)p$, then $A_i^{R_n}(M_n)$ and $A_i^{R_m}(M_m)$ are the same mod $kp$. Unfortunately, data from $p=2, n \in \{6,7\}$ shows that this is not true. However, the multisets still agree mod 2. This can be seen in Figures \ref{fig: p2n6} and \ref{fig: p2n7}. Here, the entries $j^\ell$ indicate that $j$ appears in the multiset $\ell$ times.

\begin{figure}[h]
\begin{tabular}{|Sc|Sl|Sc|Sc|Sc|}
    \hline
    $i$ & \multicolumn{1}{c|}{$A_i(M_6)$} & $A_i(M_6) \mod 2$ & $A_i(M_6)\mod 4$ & $A_i(M_6)\mod 6$ \\
    \hline
    $0$ & $0,1,3,5,6,6,7,8,10,15$ & $0^5, 1^5$ & $0^2, 1^2, 2^3, 3^3$ & $0^3, 1^2,2,3^2,4,5$ \\
    \hline
    \multirow{2}{*}{$1$} & $1,3,4,5,6,6,6,7,7,8,8,9,10,$ & \multirow{2}{*}{$0^9,1^{10}$} & \multirow{2}{*}{$0^4,1^4,2^5,3^6$} & \multirow{2}{*}{$0^4, 1^4, 2^2, 3^3, 4^3, 5^3$} \\
    \, & $10,11,11,12,13,15$ & \, & \, & \, \\
    \hline
    \multirow{2}{*}{$2$} & $4,6,7,8,9,9,10,10,11,11,12,$ & \multirow{2}{*}{$0^8,1^7$} & \multirow{2}{*}{$0^4,1^3,2^4,3^4$} & \multirow{2}{*}{$0^3,1^2,2^2,3^3,4^3,5^2$} \\
    \, & $12,13,14,15$ & \, & \, & \, \\
    \hline
    $3$ & $9,10,12,14,15,15$ & $0^3,1^3$ & $0,1, 2^2, 3^2$ & $0,2,3^3,4$\\
    \hline
    $4$& $15$ & $1$ & $3$ & $3$\\
    \hline
\end{tabular}
\caption{Degree shifts appearing in the $\F_2[e_1,\ldots,e_6]$-resolution of $\F_2[x_1,\ldots,x_6]_{\fS_6}$.}\label{fig: p2n6}
\end{figure}

\begin{figure}[h]
\begin{tabular}{|Sc|Sl|Sc|Sc|Sc|}
    \hline
    $i$ & \multicolumn{1}{c|}{$A_i(M_7)$} & $A_i(M_7) \mod 2$ & $A_i(M_7)\mod 4$ & $A_i(M_7)\mod 6$ \\
    \hline
    $0$ & $0,3,5,6,7,9,10,12,14,21$ & $0^5, 1^5$ & $0^2, 1^3, 2^3, 3^2$ & $0^3, 1,2,3^3,4,5$ \\
    \hline
    \multirow{2}{*}{$1$} & $3,5,6,7,8,9,9,10,11,12,12,$ & \multirow{2}{*}{$0^9,1^{10}$} & \multirow{2}{*}{$0^4,1^6,2^5,3^4$} & \multirow{2}{*}{$0^3, 1^3, 2^3, 3^4, 4^3, 5^3$} \\
    \, & $13,14,14,16,17,19,21$ & \, & \, & \, \\
    \hline
    \multirow{2}{*}{$2$} & $8,9,10,11,12,13,14,14,15$ & \multirow{2}{*}{$0^8,1^7$} & \multirow{2}{*}{$0^4,1^4,2^4,3^3$} & \multirow{2}{*}{$0^2,1^2,2^3,3^3,4^3,5^2$} \\
    \, & $16,16,17,18,19,21$ & \, & \, & \, \\
    \hline
    $3$ & $14,15,16,18,21,21$ & $0^3,1^3$ & $0,1^2, 2^2, 3$ & $0,2,3^3,4$\\
    \hline
    $4$ & $21$ & $1$ & $1$ & $3$\\ \hline
\end{tabular}
\caption{Degree shifts appearing in the $\F_2[e_1,\ldots,e_7]$-resolution of $\F_2[x_1,\ldots,x_7]_{\fS_7}$.}\label{fig: p2n7}
\end{figure}

Conjecture \ref{conj: betti number} is purely a numerical statement, but from Theorem \ref{thm: main theorem intro} it is clear that the resolutions of $M_n = I^{\fS_n}_{\Z_{(p)}}$ are related algebraically. We exhibit this relationship via a change of rings.

\begin{construction}
Let $p+1\leq n \leq 2p-1$. Consider the rings $R = \Z_{(p)}[e_1,\ldots,e_n]$ and $S=\Z_{(p)}[e_1,\ldots,e_{n-1}]$. Define a homomorphism of $\Z_{(p)}$-algebras $f:R\to S$ by
$$
f(e_j) = \begin{cases}
e_j & \text{if }j\neq n \\
e_{n-p}e_p - e_{n-p} & \text{if }j=n
\end{cases}.
$$
We give $R,S$ a $(\Z/p\Z)$-grading by setting $\deg(e_i) = i\mod p$. Then $f$ is homogeneous with respect to this grading and $f(I^{\fS_n}) = I^{\fS_{n-1}}$. We give $S$ an $R$-module structure via the multiplication map $$R\times S\to S, \quad(r,s)\mapsto f(r)s.$$

In particular, the top-degree generator $e_{n-p}e_p - e_n$ of $I^{\fS_n}$ acts on an element $s\in S$ as 
\begin{equation}\label{eq:change of rings action}
(e_{n-p}e_p-e_n)\cdot s = f(e_{n-p}e_p-e_n)s = (e_{n-p}e_p-(e_{n-p}e_p-e_{n-p}))s=e_{n-p}s.
\end{equation}
Let $K_\bullet$ be the Koszul complex which resolves $R/I^{\fS_n}$ over $R$. By (\ref{eq:change of rings action}), the differentials in the complex of $S$-modules $K_\bullet\otimes_R S$ are obtained from $K_\bullet$ by replacing all matrix entries $e_{n-p}e_p - e_n$ with $e_{n-p}$; in other words, $K_\bullet \otimes_R S$ is a Koszul complex of $S$-modules on the generators of $I^{\fS_{n-1}}$. Because the generators of $I^{\fS_{n-1}}$ form an $S$-regular sequence, the complex $K_\bullet\otimes_R S$ remains exact; moreover it resolves $S/I^{\fS_{n-1}}$ over $S$.
\end{construction}

\begin{question}
Let $kp+1\leq n < (k+1)p$ and let $C_\bullet$ be a minimal free resolution of $R/I^{\fS_n}$ over $R=\Z_{(p)}[e_1,\ldots,e_n]$. Does there exist a map of $\Z/p\Z$-graded $\Z_{(p)}$-algebras $f:R\to S = \Z_{(p)}[e_1,\ldots,e_{n-1}]$ such that $C_\bullet\otimes_R S$ resolves $S/I^{\fS_{n-1}}$ over $S$?
\end{question}

\appendix
\section{Minimal free resolutions with local coefficient rings}\label{section: appendix}

In this section, we verify that minimal free resolutions over the ring $\Z_{(p)}[e_1,\ldots,e_n]$ are unique up to isomorphism. We closely follow the structure of the proof of this fact for polynomial rings $\kk[x_1,\ldots,x_n]$, where $\kk$ is a field, found in \cite{Peeva}. We modify the arguments taking inspiration from proofs involving the local case (see, e.g., \cite{matsumura}).

Let $(A, \aa, K)$ be a Noetherian local ring. Consider the polynomial ring $R = A[z_1,\ldots,z_n]$, which we make a graded $A$-algebra by setting $\deg(a) = 0$ for all $a\in A$ and $\deg(z_i)=d_i>0$ for all $1\leq i\leq n$. Set $\mm = \aa R + (z_1,\ldots,z_n)R\subset R$. This is the unique homogeneous maximal ideal of $R$.

\begin{lemma}[Generalized graded Nakayama Lemma]\label{lem: upgraded Nakayama}
    Let $U$ be a finitely generated graded $R$-module and let $J\subset R$ be a proper homogeneous ideal. Then the following hold:
    \begin{enumerate}
        \item[(1)] if $JU = U$ then $U=0$, and
        \item[(2)] if $W\subset U$ is a graded $R$-submodule with $U = W + JU$, then $U = W$.
    \end{enumerate}
\end{lemma}

\begin{proof}
    First we show (1). Assume that $JU=U$ and $U$ is nonzero. Fix a finite system $\cG$ of homogeneous minimal $R$-module generators for $U$. Let $m$ be an element of $\cG$ of minimal degree. Then $U_j = 0$ for $j < \deg(m)$. Every element of $JU$ is either of larger degree than $\deg(m)$, or, since $J$ is a proper ideal, must lie in $(JU)_{\deg(m)} = J_0 \cdot U_{\deg(m)} \subset \aa R \cdot U_{\deg(m)}$. By assumption, $m\in JU$, hence $m\in \aa R\cdot U_{\deg(m)}$. Fix a subset $\cG'$ of $\cG$ that minimally generates $U_{\deg(m)}$. Then we can write $m = \sum_{m'\in\cG'} a_{m'}m'$, where $a_{m'}\in \aa$. Since $m$ is a minimal generator, it must appear on the right-hand side with nonzero coefficient, hence we have
    \begin{align*}
        m - a_m m &= \sum_{m \neq m'\in \cG'}a_{m'}m' \\
        (1-a_m) m &= \sum_{m\neq m'\in \cG'}a_{m'}m'.
    \end{align*}
    But $a_m\in \aa = \rad(A)$, hence $1-a_m$ is a unit in $A$, and consequently is also a unit in $R$. This contradicts minimality of $m$ as a generator of $U$, hence $U=0$.
    
    (2) follows by applying (1) to the graded $R$-module $U/W$.
\end{proof}

\begin{thm}[Analogue to Foundational Theorem 2.12 in \cite{Peeva}]\label{thm: foundational}
    Let $U$ be a finitely generated graded $R$-module and set $\overline{U}:= U/\mm U$. Then $\overline{U}$ is a finite dimensional graded $K$-vector space. Let $p = \dim_K\overline{U}$.
    \begin{enumerate}
        \item[(1)] Let $\{\overline{u}_1, \ldots,\overline{u}_p\}$ be a homogeneous basis for $\overline{U}$. For each $1\leq i \leq p$, choose a homogeneous preimage $u_i\in U$ of $\overline{u}_i$. Then $\{u_1,\ldots,u_p\}$ is a minimal homogeneous system of generators for $U$.
        \item[(2)] Every minimal system of homogeneous generators of $U$ is obtained as in (1).
        \item[(3)] Every minimal system of homogeneous generators of $U$ has $p$ elements. Set $q_i = \dim_K(\overline{U}_i)$ for each $i$. Then every minimal system of homogeneous generators of $U$ contains $q_i$ elements of degree $i$.
        \item[(4)] Let $\{u_1,\ldots,u_p\}$ and $\{v_1,\ldots,v_p\}$ be two minimal systems of homogeneous generators of $U$, and let $v_s = \sum_j r_{js}u_j$ with $r_{js}\in R$ for each $s$. For all $s,j$ set $c_{js}$ to be the homogeneous component of $r_{js}$ of degree $\deg(v_s) - \deg(u_j)$. Then the following three properties hold: $v_s = \sum_j c_{js}u_j$ for all $s$, $\det([c_{js}])\in A^\times$, and $[c_{js}]$ is an invertible matrix with homogeneous entries.
    \end{enumerate}
\end{thm}

\begin{proof}
    To show (1), first note that $U = \mm U + Ru_1 + \cdots + Ru_p$. By Lemma \ref{lem: upgraded Nakayama} (2), we have that $U = Ru_1 + \cdots +Ru_p$, hence $\{u_1,\ldots,u_p\}$ generates $U$. If this is not a minimal generating set, then there is some relation (possibly after renumbering) of the form $u_1 = \alpha_2u_2 + \cdots + \alpha_p u_p$, for $\alpha_i\in R$. Descending to $\overline{U}$ gives a relation $\overline{u}_1 = \overline{\alpha}_2\overline{u}_2 + \cdots + \overline{\alpha}_p\overline{u}_p$, where $\overline{\alpha}_i$ is the image of $\alpha_i$ in $\overline U$. This contradicts that $\{\overline{u}_1,\ldots,\overline{u}_p\}$ is a $K$-basis over $\overline{U}$, hence $\{u_1,\ldots,u_p\}$ must minimally generate $U$.
    
    To prove (2), assume that $\{u_1,\ldots,u_p\}$ is a minimal system of homogeneous generators of $U$. Then $\{\overline{u}_1,\ldots,\overline{u}_p\}$ generates $\overline{U}$. If there is a linear dependence among $\{\overline{u}_1,\ldots,\overline{u}_p\}$, then choose a proper subset $\{\overline{u}_{i_1},\ldots,\overline{u}_{i_q}\}$ that is a $K$-basis of $\overline{U}$. By (1), the preimages $\{u_{i_1},\ldots,u_{i_q}\}$ generate $U$ as an $R$-module, contradicting minimality of the generating set $\{u_1,\ldots,u_p\}$. Hence $\{\overline{u}_1,\ldots,\overline{u}_p\}$ must be a $K$-basis for $\overline U$.
    
    Statement (3) follows from (1) and (2).
    
    To show (4), let $\{u_1,\ldots,u_p\}$ and $\{v_1,\ldots,v_p\}$ be two minimal sets of homogeneous $R$-module generators of $U$. Assume that the generators in each set are ordered in increasing degree. By homogeneity, we know that $v_s = \sum_{j} c_{js}u_j$ for all $s$. Let $C$ be the matrix with entries $c_{js}$. For each $i$, let $B_i$ denote the $q_i\times q_i$ block on the diagonal of $C$ corresponding to the generators of degree $i$. Then $\deg(u_j)\geq \deg(v_s)$ for $j>s$, hence $c_{js} = 0$ if $j>s$ and $c_{js}$ is outside the block $B_{\deg(v_s)}$. The entries in the blocks $B_i$ are of degree 0, hence they lie in $A$, and $\det(C) = \prod_i \det(B_i)$. Let $\overline{C}$ denote the matrix with entries $\overline{c}_{js}$, where $\overline{c}_{js}$ is the image of $c_{js}$ in $K = R/\mm$. Then $\overline{C}$ is a matrix with its only nonzero entries appearing in the blocks $\overline{B}_i$. Note that $\overline{v}_s = \sum_j \overline{c}_{js} \overline{u}_j$ for all $s$, so $\overline{C}$ is a change of basis matrix for $K$. It follows that $\overline{C}$ is invertible and $\det(\overline{C})$ is a unit. Because $\det(C)= \prod_i \det(B_i)$ lies in $A$, we have $\det(C) = \det(\overline{C})+a$, where $a\in \aa$. Since $\det(\overline{C})$ is a unit, then so is $\det(C)$ since $a\in \rad(A)$. 
\end{proof}

\begin{definition}\label{def: trivial complex}
    A complex of the form 
    $$0 \to R(-p) \overset{1}{\to} R(-p) \to 0$$
    is called a \textit{short trivial complex}. A direct sum of short trivial complexes, possibly placed in different homological degrees, is called a \textit{trivial complex}.
\end{definition}

\begin{thm}[Analogue to Theorem 7.5(2) of \cite{Peeva}]\label{thm: peeva 7.5(2)}
    Let $U$ be a finitely generated graded $R$-module. Let $\mathbf{F}$ be a minimal graded free resolution of $U$, and let $\mathbf{G}$ be a graded free resolution of $U$. Then $\mathbf{G}\cong \mathbf{F}\oplus\mathbf{T}$ as complexes, where $\mathbf{T}$ is some trivial complex.
\end{thm}

\begin{proof}
    By \cite[Lemma 6.7]{Peeva}, the identity map $\id_U:U \to U$ induces graded maps of complexes $\varphi:\mathbf{F}\to \mathbf{G}$ and $\psi:\mathbf{G}\to\mathbf{F}$ having degree 0. Moreover, there exists a graded homotopy $h$ of internal degree 0 such that
    $$\id_i - \psi_i\varphi_i = d_{i+1}h_i + h_{i-1}d_i:F_i \to F_i$$
    for each $i$. Since $\mathbf{F}$ is minimal, we can repeatedly apply the fact that $\Im(d_i) \subseteq \mm F_{i-1}$ for each $i$ to obtain that $\Im(\id_i - \psi_i\varphi_i) \subseteq \mm F_i$.
    
    Choose a homogeneous basis for $F_i$, ordering it so that the degrees of the basis elements increase. Let $C = [c_{rj}]$ be the matrix of $\psi_i\varphi_i$ with respect to this ordered basis. Then $C$ has square blocks $B_j$ along the diagonal with entries in $A$, and all entries below the blocks are zero. The matrix of $\id_i - \psi_i\varphi_i$ is $E-C$, where $E$ is the identity matrix of the correct dimension. Since $\Im(\id_i - \psi_i\varphi_i)\subseteq \mm F_i$, the matrix $E-C$ has entries in $\mm$. Thus, since the diagonal entries of $E$ are 1, the diagonal entries of $C$ must also be 1. The remaining entries in the blocks $B_j$ must lie in $\aa$, else $E-C$ would have entries that are units in $R$. We have $\det(C) = \prod_j \det(B_j)$. Modding out by $\mm$, we have that $\overline{C}$ must be the identity matrix, hence has determinant 1. But $\det(\overline{C}) = \prod_j \det(\overline{B}_j)$, so each $\overline{B}_j$ has determinant a nonzero element of $K$. Hence, $\det(C) = \prod_j (\det(\overline{B}_j) + a_j)$, where $a_j\in\aa$, so this is invertible in $A\subset R$. Thus, $\psi\varphi:\mathbf{F}\to \mathbf{F}$ is an isomorphism. Let $\xi:\mathbf{F}\to\mathbf{F}$ be its inverse. Then 
    $$\mathbf{F}\overset{\varphi}{\to}\mathbf{G}\overset{\xi\psi}{\to}\mathbf{F}$$
    is a splitting. Write $\mathbf{T} = \ker(\xi\psi)$. Then $\mathbf{G}\cong \varphi(\mathbf{F}) \oplus \mathbf{T}$ as graded modules. It remains to show that this is an isomorphism of chain complexes and that $\mathbf{T}$ is a trivial complex. The rest of the proof (see \cite[p.37]{Peeva}) does not depend on the coefficient ring of $R$, hence we omit it.
\end{proof}

\newpage

\bibliographystyle{amsplain}
\bibliography{biblio}

\end{document}